\documentclass[a4paper,10pt]{article}

\usepackage{latexsym,amstext,amsmath,amssymb,amsthm,bbm}
\usepackage[all]{xy}
\pdfpagewidth=\paperwidth
\pdfpageheight=\paperheight

\title{Multidegree for bifiltered $D$-modules}

\author{R\'emi Arcadias\\
arcadias@lab.twcu.ac.jp\\
Tokyo Woman's Christian University,\\
2-6-1 Zempukuji, Suginami-ku, Tokyo 167-8585, Japan.
}

\newcommand{\D}{\mathcal{D}}
\newcommand{\C}{\mathbb{C}}

\newtheorem{prop}{Proposition}[section]
\newtheorem{theo}{Theorem}[section]
\newtheorem{definition}{Definition}[section]
\newtheorem{lemme}{Lemma}[section]

\newtheorem{remark}{Remark}[section]

\begin{document}

\maketitle

\begin{abstract}

In commutative algebra, E. Miller and B. Sturmfels defined the notion of multidegree for multigraded modules over a multigraded polynomial ring. 
We apply this theory to bifiltered modules over the Weyl algebra $D$. The bifiltration is a combination of the standard filtration by the order of differential operators and of the so-called $V$-filtration along a coordinate subvariety of the ambient space defined by M. Kashiwara. The multidegree we define provides a new invariant for $D$-modules. We investigate its relation with the $L$-characteristic cycles considered by Y. Laurent.
We give examples from the theory of $A$-hypergeometric systems $M_A(\beta)$ defined by I. M. Gelfand, M. M. Kapranov and A. V. Zelevinsky. We consider the $V$-filtration along the origin. When the toric projective variety defined from the matrix $A$ is Cohen-Macaulay, we have an explicit formula for the multidegree of $M_A(\beta)$ .

\end{abstract}

\section*{Introduction}

We consider finite type modules over the Weyl algebra 
\[
D=\C[x_1,\dots,x_n]\langle \partial_1,\dots,\partial_n\rangle.
\]
It is classical to endow $D$ with the filtration by the order in $\partial_1,\dots,\partial_n$, which we call the $F$-filration, and to endow a $D$-module $M$ with a good $F$-filtration. For instance that leads to the notion of the characteristic variety, which is the support of $\textrm{gr}^F(M)$, and to the characteristic cycle.
M. Kashiwara introduced another kind of filtration, the $V$-filtration along a smooth subvariety $Y$ of $\C^n$. Then one has the notion of a good $(F,V)$-bifiltration (c.f.\ \cite{laurent88}), and we can also consider intermediate filtrations $L$ between $F$ and $V$ as developed by Y. Laurent in his theory of slopes (c.f.\ \cite{laurent04}). This leads to $L$-characteristic varieties (the support of $\textrm{gr}^L(M)$) and $L$-characteristic cycles. 

Exploring that theory with homological methods, M. Granger, T. Oaku and N. Takayama considered $(F,V)$-bifiltered free resolutions of finite type $D$-modules in \cite{granger04}, \cite{oaku01}. More precisely, dealing with local analytic $D$-modules, they can define minimal bifiltered free resolutions. That provides invariants attached to a bifiltered module: the ranks, also called Betti numbers, and the shifts appearing in the minimal resolution. 
In the category of modules over the global Weyl algebra, $(F,V)$-bifiltered free resolutions still can be considered, but the minimality no longer makes sense. 

Our main purpose in this paper is to introduce a new invariant, the multidegree, derived from the Betti numbers and shifts arising from any bifiltered free resolution of a $(F,V)$-bifiltered $D$-module. It will be independent of the good bifiltration, i.e.\ a chosen presentation of the module. We will relate this invariant to the $L$-characteristic cycles.

To achieve this, we use the theory of $K$-polynomial and multidegree, as was developed by E. Miller and B. Sturmfels in \cite{stu05}. The multidegree is a generalization of the usual degree in projective geometry; it is defined for finite type multigraded modules over a polynomial ring. After reviewing this theory in Section 1, we adapt it first to $F$-filtered $D$-modules in Section 2. We obtain the notion of multidegree for a $F$-filtered $D$-module, which is independent of the good filtration. This multidegree is a monomial $mT^d$ with $m\in\mathbb{N}$; we interpret $m$ and $d$ as a generic multiplicity and a generic codimension respectively.

Then we adapt the theory of multidegree to $(F,V)$-bifiltered $D$-modules in section 3. 
The multidegree is an element of $\mathbb{Z}[T_1,T_2]$, denoted by $\mathcal{C}_{F,V}(M;T_1,T_2)$, homogeneous in $T_1,T_2$. Its degree $d$ has to be fixed because of the non-positivity of the multigrading considered: if $Y$ is the origin in $\C^n$, $d$ is the codimension of the $V$-homogenization module $\mathbf{R}_V(M)$. Using a proof in \cite{laurent88}, we can show that $\mathcal{C}_{F,V}(M;T_1,T_2)$ is an invariant attached to the module, indepedently of the good bifiltration. 

In section 4, we assume a strong regularity condition on the $(F,V)$-bifiltered module, which we call a nicely bifiltered module. We prove that in the holonomic case, this condition implies that the module has no slopes along $Y$. Then we show that the multidegree of such a module almost only depends on the $L$-characteristic cycle of the module, with $L$ an intermediate filtration close to $F$ or close to $V$. Let us note here that we have to deal with some codimensions which may alter the link between mutidegree and $L$-characteristic cycle: the codimension of the module $\mathbf{R}_V(M)$ may not be equal to that of $\textrm{gr}^L(M)$.

Finally, we use the theory of hypergeometric systems to provide interesting examples in section 5. We consider the hypergeometric module $M_A(\beta)$ introduced by 
I. M. Gelfand, M. M. Kapranov and A. V. Zelevinsky in \cite{GKZ}, in the case where the semigroup generated by the columns $a_1,\dots,a_n$ of the matrix $A$ is pointed. We take $Y$ to be the origin in $\C^n$. In that case 
the problems about codimensions described above does not remain, and the multidegree only depends on the $L$-characteristic cycle if $M_A(\beta)$ is nicely bifiltered. Let  $\textrm{vol}(A)$ denotes the normalized volume of the convex hull of the set $\{0,a_1,\dots,a_n\}$ in $\mathbb{R}^d$.
Let us assume that the closure in $\mathbb{P}^n$ of the variety defined by $I_A$ is Cohen-Macaulay. Then for generic parameters $\beta$ (or for all parameters if $I_A$ is homogeneous),
niceness holds and we have:
\[
\mathcal{C}_{F,V}(M_A(\beta);T_1,T_2)=
\textrm{vol}(A).\sum_{j=d}^{n}\binom{n-d}{j-d}T_1^jT_2^{n-j}.
\]
We give examples, computed with the computer algebra systems Singular \cite{singular} and Macaulay2 \cite{M2}. 

\section*{Acknowledgments}

I sincerely thank L. Narv\'aez-Macarro for asking about a link between Betti numbers of $D$-modules and classical invariants in that theory, such as the characteristic cycle. 
I am also grateful to T. Oaku, who advised me throughout this work; in particular he suggested to use the arguments in \cite{laurent88} to prove the invariance of the multidegree.
I finally thank the Japan Society for the Promotion of Science for the financial support. 

\section{Multidegree for modules over a commutative polynomial ring}

\subsection{Review of the theory}

Let us give a review of the theory of $K$-polynomials and multidegrees in the commutative setting. Let $S=k[x_1,\dots,x_n]$ with $k$ a field. A multigrading on $S$ is given by a homomorphism of abelian groups 
$\textrm{deg} : \mathbb{Z}^n\to \mathbb{Z}^d$ with, denoting by $e_1,\dots,e_n$ the canonical base of $\mathbb{Z}^n$, $\textrm{deg}(e_i)=a_i\in \mathbb{Z}^d$. Identifying the set of monomials of $S$ with $\mathbb{N}^n$, we have $\textrm{deg}(x_1^{\alpha_1}\dots x_n^{\alpha_n})=\sum \alpha_i a_i$, and $S$ becomes a multigraded ring over $\mathbb{Z}^d$. 

Let $M=\bigoplus_{a\in\mathbb{Z}^d}M_a$ be a multigraded $S$-module of finite type. 
For $b\in\mathbb{Z}^d$, let us denote by $S[b]$ the module $S$ endowed with the multigrading such that for any $a\in\mathbb{Z}^d$, $S[b]_a=S_{a-b}$. A multigraded free module is a module isomorphic to $\bigoplus_{j=1}^{r}S[b_j]$, with $b_1,\dots,b_r\in\mathbb{Z}^d$.

Take a multigraded free resolution, i.e. a multigraded exact sequence
\[
 0 \to \mathcal{L}_{\delta}\to \cdots \to \mathcal{L}_1\to \mathcal{L}_0\to M\to 0,
\]
with $\mathcal{L}_i$ a multigraded free module.

\begin{definition}
For $b=(b_1,\dots,b_d)\in\mathbb{Z}^d$, 
the $K$-polynomial of $S[b]$ is defined by 
\[
K(S[b];T_1,\dots,T_d)=T_1^{b_1}\dots T_d^{b_d} \in \mathbb{Z}[T_1,\dots,T_d,T_1^{-1},\dots,T_d^{-1}].
\]
For $b_1,\dots,b_r\in\mathbb{Z}^d$,
The $K$-polynomial of $\mathcal{L}=\bigoplus_{j=1}^{r}S[b_j]$ is defined by 
\[
K(\mathcal{L};T_1,\dots,T_d)=\sum_j K(S[b_j];T_1,\dots,T_d) \in \mathbb{Z}[T_1,\dots,T_d,T_1^{-1},\dots,T_d^{-1}].
\]
Then the $K$-polynomial of $M$ is defined by 
\[
K(M;T)=\sum_i (-1)^i K(\mathcal{L}_i;T_1,\dots,T_d) \in \mathbb{Z}[T_1,\dots,T_d,T_1^{-1},\dots,T_d^{-1}].
\]
\end{definition}

\begin{prop}[\cite{stu05}, Theorem 8.34]\label{prop1}
The definition of $K(M;T_1,\dots,T_d)$ does not depend on the multigraded free resolution.
\end{prop}

If we substitute $T_1,\dots,T_d$ by $1-T_1,\dots,1-T_d$ in $K(M;T_1,\dots,T_d)$, we get a well-defined power series in $\mathbb{Z}[[T_1,\dots,T_d]]$. We then consider the total degree in $T_1,\dots,T_d$.

\begin{definition}
We denote by $\mathcal{C}(M;T_1,\dots,T_d)\in\mathbb{Z}[T_1,\dots,T_d]$ the sum of the terms whose total degree 
equals $\mathrm{codim} M$ in 
$K(M;1-T_1,\dots,1-T_d)$. This is called the \emph{multidegree} of $M$.
\end{definition}

Remind that the module $M$ defines an algebraic cycle $\sum m_{i}Z_{i}$, where $Z_{i}$, defined by ideals $\mathfrak{p}_i$, are the irreducible components of $\textrm{rad}(\textrm{ann} M)$ and $m_{i}$ is the multiplicity of $M_{\mathfrak{p}_i}$. It turns out that the multidegree depends only on the algebraic cycle.

\begin{prop}[\cite{stu05}, Theorem 8.53]\label{prop2}
If $\mathfrak{p}_1,\dots,\mathfrak{p}_k$ are the maximal dimensional associated primes of $M$, then
\[
\mathcal{C}(M;T_1,\dots,T_d)=\sum_i (\mathrm{mult}_{\mathfrak{p}_i} M) \mathcal{C}(S/\mathfrak{p}_k;T_1,\dots,T_d).
\]
\end{prop}

$S$ is said to be positively multigraded if moreover for any $b\in\mathbb{Z}^d$, we have $\textrm{dim}_k S_b<\infty$. In that case we can consider the Hilbert series
\[
H(M;T_1,\dots,T_d)=\sum_{b\in\mathbb{Z}^d} (\textrm{dim}_k M_b) T_1^{b_1}\dots T_d^{b_d}\in \mathbb{Z}[[T_1,\dots,T_d]].
\] 
If $b=(b_1,\dots,b_d)\in\mathbb{Z}^d$, let us denote by $T^b$ the product $T_1^{b_1}\dots T_d^{b_d}$.

\begin{prop}\label{prop3}
Let $S$ be positively multigraded. Then
\begin{enumerate}
\item
\[
H(M;T_1,\dots,T_d)=\frac{K(M;T_1,\dots,T_d)}{\Pi (1-T^{a_i})}
\]
\item If $M\neq 0$, then $\mathcal{C}(M;T_1,\dots,T_d)\neq 0$, moreover $\mathcal{C}(M;T_1,\dots,T_d)$ is the sum of the non-zero terms of least total degree in $K(M;1-T_1,\dots,1-T_d)$.
\end{enumerate}
\end{prop}

The assertion 1 is \cite{stu05}, Theorem 8.20, and the assertion 2 follows from \cite{stu05}, Claim 8.54 and Exercise 8.10.

\subsection{Genericity}

Let $S=k[\lambda_1,\dots,\lambda_p][x_1,\dots,x_n]$ be multigraded by $\textrm{deg}\,x_i=a_i\in\mathbb{Z}^d$ and $\textrm{deg}\,\lambda_i=0$. 
We consider $\lambda_1,\dots,\lambda_p$ as parameters and study the behaviour of the $K$-polynomial under the specialization. 

Let $\mathbb{K}=\textrm{Frac}(k[\lambda_1,\dots,\lambda_p])$. Let $M=S^r/N$ be a multigraded finite type $S$-module. For $c\in k^p$, let
\[
M^c=\frac{S}{\langle\lambda_1-c_1,\dots,\lambda_p-c_p\rangle}\otimes M,
\]
considered as a multigraded $k[x_1,\dots,x_n]$-module. We are going to state that if $c$ is  generic, then $K(\mathbb{K}\otimes M;T)=K(M^c;T)$.
More precisely, we shall describe the exceptional values of $c$ in terms of Gr\"obner bases.

Let $<$ be a well-ordering on $\mathbb{N}^n\times\{1,\dots,r\}$, such that for any $\alpha,\beta,\delta\in \mathbb{N}^n$ and $i,i' \in\{1,\dots,r\}$, we have
\[
(\alpha,i)<(\beta,i') \Rightarrow (\alpha+\delta,i)<(\beta+\delta,i'),
\]
and let $<'$ be the well-ordering on $\mathbb{N}^p\times\mathbb{N}^n\times\{1,\dots,r\}$ defined by
\[
(\alpha,\beta,i)<'(\alpha',\beta',i')\ \textrm{iff}\
\left\{
\begin{array}{l}
(\beta,i)<(\beta',i')\\
\textrm{or}\ ((\beta,i)=(\beta',i')\ \textrm{and}\ \alpha <_{\textrm{lex}}\alpha').
\end{array}
\right.
\]
Let $P_1,\dots,P_s$ be a Gr\"obner base of $N$. For $1\leq i\leq s$, 
$q_i(\lambda)\in k[\lambda]$ denotes the leading coefficient, with respect to $<$, of the image of $P_i$ in $\mathbb{K}\otimes S$. For $P\in k[x]^r$ or $P\in \mathbb{K}[x]^r$, we denote by $\textrm{Exp}_<P\in \mathbb{N}^n\times\{1,\dots,r\}$ the leading exponent of $P$ with respect to $<$.

\begin{prop}[\cite{oaku97}, Propositions 6 and 7]\label{prop9}

\begin{enumerate}
\item $P_1,\dots,P_s$ is a Gr\"obner base of $\mathbb{K}\otimes N$.
\item Let $c\in k^n$ such that $c\notin \bigcup_i (q_i=0)$.
Then $P_1(c),\dots,P_s(c)$ is a Gr\"obner base of $N^c$ and 
$\mathrm{Exp}_< \mathbb{K}\otimes N=\mathrm{Exp}_< N^c$.
\end{enumerate}
\end{prop}

\begin{prop}\label{prop10}
Let $c\in k^n$ such that $c\notin \bigcup_i (q_i=0)$. Then $K(\mathbb{K}\otimes M;T)=K(M^c;T)$.
Consequently $\mathcal{C}(\mathbb{K}\otimes M;T)=\mathcal{C}(M^c;T)$.
\end{prop}

This follows from Proposition \ref{prop9} and from \cite{stu05}, Theorem 8.36 which asserts that the $K$-polynomial remains the same when taking the initial module with respect to any well-ordering.

\section{Multidegree for $F$-filtered $D$-modules}

Let $D=\C[x_1,\dots,x_n]\langle\partial_1,\dots,\partial_n\rangle$ be the Weyl algebra. 
A vector $(\mathbf{u},\mathbf{v})\in\mathbb{Z}^n\times\mathbb{Z}^n$ is called an \emph{admissible weight vector} for $D$ if for all $i$, $u_i+v_i\geq 0$.
For $P=\sum a_{\alpha,\beta}(x)x^{\alpha}\partial^{\beta}\in D$, we define 
\[
\textrm{ord}^{(\mathbf{u},\mathbf{v})}(P)=
\textrm{max}_{(\alpha,\beta)|a_{\alpha,\beta}\neq 0}(\sum u_i\alpha_i+\sum v_i\beta_i).
\]
We then define an increasing filtration by $F^{(\mathbf{u},\mathbf{v})}_d(D)=\{P\in D, \textrm{ord}^F(P)\leq d\}$ with $d\in\mathbb{Z}$.

In this section we consider only the weight vector $(\mathbf{0},\mathbf{1})$; we will simply denote the associated filtration by $(F_d(D))_{d\in\mathbb{N}}$, called the $F$-filtration. We have $\textrm{gr}^F(D)\simeq\C[x_1,\dots,x_n,\xi_1,\dots,\xi_n]$.

Let $M$ be a $D$-module. An $F$-filtration of $M$ is an exhausting increasing filtration $(F_d(M))_{d\in \mathbb{N}}$ compatible with the $F$-filtration of $D$. For $\mathbf{n}=(n_1,\dots,n_r)\in \mathbb{Z}^r$, let us denote by $D^r[\mathbf{n}]$ the module $D^r$ endowed with the $F$-filtration such that 
$F_d(D^r[\mathbf{n}])=\bigoplus_{i=1}^r F_{d-n_i}(D)$. If $N$ is a submodule of $D^r$, we endow $D^r[\mathbf{n}]/N$ with the quotient filtration, i.e. 
\[
F_d\left(\frac{D^r[\mathbf{n}]}{N}\right)=\frac{F_d(D^r[\mathbf{n}])+N}{N}.
\]
We say that a filtration $F_d(M)$ is good if $M$ is isomorphic as an $F$-filtered $D$-module to a module of the type $D^r[\mathbf{n}]/N$.

Let us take a filtered free resolution 
\[
 0 \to D^{r_{\delta}}[\mathbf{n}^{(\delta)}]\to \cdots \to D^{r_1}[\mathbf{n}^{(1)}]\to D^{r_0}[\mathbf{n}^{(0)}]\to M\to 0.
\]
Its existence can be proved in the same way as \cite{granger04}, Theorem 3.4, forgetting the minimality.

\begin{definition}
The $K$-polynomial of $D^r[\mathbf{n}]$ is defined by 
\[
K_F(D^r[\mathbf{n}];T)=\sum_i T^{\mathbf{n}_i} \in \mathbb{Z}[T].
\]
The $K$-polynomial of $M$ is defined by 
\[
K_F(M;T)=\sum_i (-1)^i K_F(D^{r_i}[\mathbf{n}^{(i)}];T) \in \mathbb{Z}[T].
\]
\end{definition}

\begin{prop}\label{prop4}
The definition of $K_F(M;T)$ does not depend on the filtered free resolution.
\end{prop}

\begin{proof}
Let $R=\textrm{gr}^F(D)$, and for $\mathbf{n}=(n_1,\dots,n_r)$, 
$R^r[\mathbf{n}]=\oplus_{i=1}^r R[n_i]$.
By grading the filtered free resolution we get a graded free resolution over the commutative ring $R$:
\[
 0 \to R^{r_{\delta}}[\mathbf{n}^{(\delta)}]\to \cdots \to R^{r_1}[\mathbf{n}^{(1)}]\to R^{r_0}[\mathbf{n}^{(0)}]\to\textrm{gr}^F(M) \to 0.
\]
The $K$-polynomial is unchanged. Then apply Proposition \ref{prop1}.
\end{proof}

\begin{definition}
We denote by $\mathcal{C}_F(M;T)$ the term of least degree in $T$ in $K_F(M;1-T)$. This is the multidegree of $M$ with respect to $F$.
\end{definition}

\begin{prop}\label{prop5}
$\mathcal{C}_F(M;T)$ does not depend on the good filtration.
\end{prop}

\begin{proof}
Again we argue by grading. We have $\mathcal{C}_F(M;T)=\mathcal{C}(\textrm{gr}^F(M);T)$. Let $\mathbb{K}=\textrm{Frac}(\C[x])$. We have $\mathcal{C}(\textrm{gr}^F(M);T)=\mathcal{C}(\mathbb{K}\otimes\textrm{gr}^F(M);T)$. The graded ring $\mathbb{K}\otimes \textrm{gr}^F(D)$ is a positively graded ring. Hence the $K$-polynomial is equal to the numerator of the Hilbert series, by Proposition \ref{prop3}. The multidegree is of the form $mT^d$ with $d=\textrm{codim}\, \mathbb{K}\otimes\textrm{gr}^F(M)$ (unless it is $0$), and $m$ is the multiplicity of $\mathbb{K}\otimes\textrm{gr}^F(M)$ along the maximal ideal $\xi_1,\dots,\xi_n$. We can show that this data is independent of the good filtration in the same way as \cite{cimpa1}, Remark 12 and Proposition 25.
\end{proof}

Let us give some interpretation. We have $\mathcal{C}_F(M;T)=mT^d$.
For $x_0\in \C^n$, the graded $\C[\xi]$-module
$(\textrm{gr}^F(M))^{x_0}$ is defined as in the section 1.2.

\begin{prop}\label{prop6}
\begin{enumerate}
\item
 $m$ and $d$ are equal respectively to the multiplicity and the codimension of the graded $\C[\xi]$-module $\mathrm{gr}^F(M)^{x_0}$ for $x_0$ generic. Let us denote by $\pi:T^*\C^n\to \C^n$ the canonical projection. $d$ is equal to the codimension of the variety $\mathrm{char}\,M\cap \pi^{-1}(x_0)$ for $x_0$ generic.
\item If moreover $M$ is holonomic, then $m=\mathrm{rank}\,M=\mathrm{dim}_{\mathbb{K}} \mathbb{K}\otimes \mathrm{gr}^F(M)$.
\end{enumerate}
\end{prop}

\begin{proof}
\begin{enumerate}
\item This is Proposition \ref{prop10}.
\item In the holonomic case, $\mathbb{K}\otimes\textrm{gr}^F(M)$ is finite dimensional over $\mathbb{K}$, and we have 
\[
\textrm{dim}_{\mathbb{K}} \mathbb{K}\otimes \textrm{gr}^F(M)=
H(\mathbb{K}\otimes\textrm{gr}^F(M);T)_{|T=1}.
\]
The result follows, by using Proposition \ref{prop3}.
\end{enumerate}
\end{proof}

\section{Multidegree for $(F,V)$-bifiltered $D$-modules}

Now set $D=\C[x_1,\dots,x_n,t_1,\dots,t_p]\langle\partial_{x_1},\dots,\partial_{x_n},\partial_{t_1},\dots,\partial_{t_p}\rangle$. We still endow it with the $F$-filtration. We introduce the $V$-filtration along $t_1=\dots=t_p=0$. This is the filtration defined by assigning the weight vector 
$(\mathbf{0},\mathbf{-1},\mathbf{0},\mathbf{1})$ to the set of variables $(x,t,\partial_x,\partial_t)$. We denote this filtration by $(V_k(D))_{k\in\mathbb{Z}}$.

Then we have the $(F,V)$-bifiltration on $D$ defined by $F_{d,k}(D)=F_d(D)\cap V_k(D)$ for $d,k\in\mathbb{Z}$. 
For $\mathbf{n},\mathbf{m}\in \mathbb{Z}^r$, let us denote by $D^r[\mathbf{n}][\mathbf{m}]$ the module $D^r$ endowed with the bifiltration such that 
\[
F_{d,k}(D^r[\mathbf{n}][\mathbf{m}])=\bigoplus_{i=1}^r F_{d-n_i,k-m_i}(D).
\] 
A quotient $D^r[\mathbf{n}][\mathbf{m}]/N$ is endowed with the bifiltration $F_{d,k}(D^r[\mathbf{n}][\mathbf{m}]/N)=(F_{d,k}(D^r[\mathbf{n}][\mathbf{m}])+N)/N$. 

Let $M$ be a $D$-module. A good bifiltration $(F_{d,k}(M))_{d\in \mathbb{N},k\in\mathbb{Z}}$ is an exhaustive increasing bifiltration, compatible with the bifiltration $(F_{d,k}(D))$, such that $M$ is isomorphic as a bifiltered module to a module of the type  $D^r[\mathbf{n}][\mathbf{m}]/N$.

\begin{prop}\label{prop12} $M$ admits a bifiltered free resolution, i.e.\ a bifiltered exact sequence 
\[
 0 \to D^{r_{\delta}}[\mathbf{n}^{(\delta)}][\mathbf{m}^{(\delta)}]\to \cdots \to D^{r_1}[\mathbf{n}^{(1)}][\mathbf{m}^{(1)}]\to 
D^{r_0}[\mathbf{n}^{(0)}][\mathbf{m}^{(0)}]\to M\to 0.
\]
\end{prop}
We shall prove this proposition in a constructive way. For this purpose, let us introduce some Rees algebras. First, we have the Rees algebra with respect to the $F$-filtration (c.f. \cite{narvaez97}):
\[
\mathcal{R}_{F}(D)=\bigoplus_{d}F_{d}(D)\tau^{d}.
\]
This is endowed with the $V$-filtration :
\[
V_k(\mathcal{R}_F(D))=\bigoplus_{k\in\mathcal{Z}} F_{d,k}(\D) \tau^d \ \textrm{for}\, d\in\mathbb{N}.
\]
$\mathcal{R}_{F}(D)$ is isomorphic to the $\C$-algebra generated by $x_{i}$, $t_{i}$, $(\partial_{x_{i}}\tau)$, 
$(\partial_{t_{i}}\tau)$, $\tau$, subject to the relations 
$[\partial_{x_{i}}\tau,x_{i}]=\tau$ and $[\partial_{t_{i}}\tau,t_{i}]=\tau$, the commutators involving other pairs of generators being zero. This is a noetherian algebra.
We will replace respectively the generators  $x_{i}$, $t_{i}$, $\partial_{x_{i}}\tau$, 
$\partial_{t_{i}}\tau$, $\tau$ by  $x_{i}$, $t_{i}$, $\partial_{x_{i}}$, 
$\partial_{t_{i}}$, $h$, thus we identify $\mathcal{R}_{F}(D)$ with the $\C$-algebra, denoted $D^{(h)}$, generated by $x_{i}$, $t_{i}$, $\partial_{x_{i}}$, $\partial_{t_{i}}$, $h$, subject to the relations 
\[
[\partial_{x_{i}},x_{i}]=h \quad \textrm{and} \quad [\partial_{t_{i}},t_{i}]=h.
\]
An admissible weight vector for $D^{(h)}$ is a vector $(\mathbf{u},\mathbf{v},l)\in\mathbb{Z}^{n+p}\times\mathbb{Z}^{n+p}\times\mathbb{Z}$ such that for any $i$, $u_i+v_i\geq l$. A filtration is associated with such a vector by assigning it to the set of variables $(x,t,\partial_x,\partial_t,h)$.
The filtration associated with $(\mathbf{u},\mathbf{v},l)=(\mathbf{0},\mathbf{-1},\mathbf{0},\mathbf{1},0)$ gives the $V$-filtration.
The bigraded ring $\textrm{gr}^V(D^{(h)})$ is isomorphic to $D^{(h)}$
endowed with the following multigrading :
\begin{description}
\item
$\textrm{deg}(x_i)=(0,0)$, \quad $\textrm{deg}(t_i)=(0,-1)$, \quad $\textrm{deg}(h)=(1,0)$,
\item
$\textrm{deg}(\partial_{x_i})=(1,0)$,\quad
$\textrm{deg}(\partial_{t_i})=(1,1)$.
\end{description}

Let us denote $F_{d}(M)=\bigcup_{k}F_{d,k}(M)$. 
We associate with $M$ a $\mathcal{R}_F(D)$-module $\mathcal{R}_F(M)=\oplus_{d}F_{d}(M)\tau^{d}$, this is endowed with a $V$-filtration
$V_{k}(\mathcal{R}_F(M))=\oplus_{d}F_{d,k}(M)\tau^{d}$.

Conversely, there exists a dehomogenizing functor $\rho_F$ (see \cite{granger04}, where this functor is denoted by $\rho$), from the category of $V$-filtered graded $D^{(h)}$-modules to the category of bifiltered $D$-modules. A $D^{(h)}$-module is said to be $h$-saturated if the action of $h$ on this module is injective. \cite{granger04}, Proposition 3.6 states that the functors $\rho_F$ and $\mathcal{R}_F$ give an equivalence of categories between the category of $h$-saturated $D^{(h)}$-modules with good $V$-filtrations and the category of $D$-modules with good bifiltrations, and that moreover these functors are exact.

We have also the Rees algebra of $D$ with respect to $V$ :
\[
\mathcal{R}_V(D)=\bigoplus_{k\in\mathbb{Z}} V_k(\D) \theta^k
\]
This is endowed with the following filtration :
\[
F_d(\mathcal{R}_V(D))=\bigoplus_{k\in\mathbb{Z}} F_{d,k}(\D) \theta^k \ \textrm{for}\, d\in\mathbb{N}
\]
$\mathcal{R}_V(D)$ is generated as a $\C$-algebra by $x_i\theta^0, \partial_{x_i}\theta^0, t_i\theta^{-1},\partial_{t_i}\theta, \theta$.
Let us denote respectively those elements by $\tilde{x_i}, \tilde{\partial_{x_i}}, \tilde{t_i},\tilde{\partial_{t_i}}, \theta$. 
The following lemma is clear.

\begin{lemme}\label{lemme1}
$\mathcal{R}_V(D)$ is isomorphic to the algebra $\C[\tilde{x_i}, \tilde{t_i},\theta]\langle\tilde{\partial_{x_i}},\tilde{\partial_{t_i}}\rangle$ subject to the relations $[ \tilde{\partial_{x_i}},\tilde{x_i}]=1$ and $[ \tilde{\partial_{t_i}},\tilde{t_i}]=1$ for any $i$. \\
The $F$-filtration is then given by assigning the weight vector $(\mathbf{0},\mathbf{0},0,\mathbf{1},\mathbf{1})$ to the set of variables $(\tilde{x},\tilde{t},\theta,\tilde{\partial_{x}},\tilde{\partial_{t}})$.
\end{lemme}


Then the bigraded ring $\textrm{gr}^F(\mathcal{R}_V(D))$ is isomorphic to the commutative polynomial ring $\C[\tilde{x_i}, \tilde{t_i},\theta,\tilde{\partial_{x_i}},\tilde{\partial_{t_i}}]$ endowed with the following multigrading :
\begin{description}
\item
$\textrm{deg}(\tilde{x_i})=(0,0)$, \quad $\textrm{deg}(\tilde{t_i})=(0,-1)$, \quad $\textrm{deg}(\theta)=(0,1)$,
\item
$\textrm{deg}(\tilde{\partial_{x_i}})=(1,0)$,\quad
$\textrm{deg}(\tilde{\partial_{t_i}})=(1,1)$.
\end{description}

Similarly, we define the Rees module associated with $M$ with respect to $V$:
\[
\mathcal{R}_V(M)=\bigoplus_{k\in\mathbb{Z}} V_k(M) \theta^k
\]
where $V_{k}(M)=\bigcup_{d}F_{d,k}(M)$. It admits an $F$-filtration
\[
F_{d}(\mathcal{R}_V(M))=\bigoplus_{k\in\mathbb{Z}} F_{d,k}(M) \theta^k
\]
such that $\textrm{gr}^{F}(\mathcal{R}_{V}(M))$ is isomorphic to
\[
\bigoplus_{d,k} \frac{F_{d,k}(M)}{F_{d-1,k}(M)} \theta^k.
\]
Conversely, as it has been stated before, there exists a dehomogenizing functor $\rho_V$, from the category of $F$-filtered graded $\mathcal{R}_V(D)$-modules to the category of bifiltered $D$-modules. 
 A $\mathcal{R}_V(D)$-module is said to be $\theta$-saturated if the action of $\theta$ on this module is injective.
The functors $\rho_V$ and $\mathcal{R}_V$ give an equivalence of categories between the category of $\theta$-saturated $\mathcal{R}_V(D)$-modules with good $F$-filtrations and the category of $D$-modules with good bifiltrations. Moreover these functors are exact.

\begin{proof}[Proof of Proposition \ref{prop12}.]
$\mathcal{R}_F(M)$ is a finite type $D^{(h)}$-module isomorphic as a $V$-filtered graded $D^{(h)}$-module to a quotient of $(D^{(h)})^r[\mathbf{m}]$. A presentation of $\mathcal{R}_F(M)$ can be obtained by means of $F$-adapted Gr\"obner bases. By replacing $D$ by $D^{(h)}$ in \cite{oaku01b}, section 3, we can construct a $V$-adapted free resolution of $\mathcal{R}_F(M)$. Dehomogenizing this resolution provides a bifiltered free resolution of $M$.

We can use also the $V$-homogenization. Using \cite{oaku01b}, section 3, we construct a presentation of $\mathcal{R}_V(M)$. We take a bigraded free resolution of $\textrm{gr}^F\mathcal{R}_V(M)$, which can be lifted to a $F$-adapted resolution of $\mathcal{R}_V(M)$, as in \cite{granger04}, Proposition 2.7. Taking $\rho_V$ gives a bifiltered free resolution of $M$. 
\end{proof}

\begin{definition}
The $K$-polynomial of $D^r[\mathbf{n}][\mathbf{m}]$ with respect to $(F,V)$ is defined by 
\[
K_{F,V}(D^r[\mathbf{n}][\mathbf{m}];T_1,T_2)=\sum_i T_1^{\mathbf{n}_i} T_2^{\mathbf{m}_i}\in \mathbb{Z}[T_1,T_2,T_2^{-1}].
\]
The $K$-polynomial of $M$ with respect to $(F,V)$ is defined by 
\[
K_{F,V}(M;T_1,T_2)=\sum_i (-1)^i K_{F,V}(D^{r_i}[\mathbf{n}^{(i)}][\mathbf{m}^{(i)}];T_1,T_2) \in \mathbb{Z}[T_1,T_2,T_2^{-1}].
\]
\end{definition}

\begin{prop}\label{prop7}
The definition of $K_{F,V}(M;T_1,T_2)$ does not depend on the bifiltered free resolution.
\end{prop}

\begin{proof}[Proof of Proposition \ref{prop7}]
A bifiltered free resolution of $M$ induces a bigraded free resolution of $\textrm{gr}^{F}(\mathcal{R}_{V}(M))$. Thus $K_{F,V}(M;T_1,T_2)=K(\textrm{gr}^{F}(\mathcal{R}_{V}(D));T_1,T_2)$ and we can apply Proposition \ref{prop1}.
\end{proof}

Let $\mathbb{K}=\textrm{Frac}(\C[x_{1},\dots,x_{n}])$. Instead of $D$, we shall work with $\mathbb{K}\otimes D$. This has no influence on the bifiltration.

\begin{definition}
We denote by $\mathcal{C}_{F,V}(M;T_1,T_2)$ the sum of the terms whose total degree in $T_1,T_2$ equals 
$\textrm{codim}\,(K\otimes\textrm{gr}^{F}(\mathcal{R}_{V}(M)))$ in the expansion of 
$K_{F,V}(M;1-T_1,1-T_2)$. This is the multidegree of $M$ with respect to $(F,V)$.
\end{definition}

\begin{theo}\label{theo1}
$\mathcal{C}_{F,V}(M;T_1,T_2)$ does not depend on the good bifiltration.
\end{theo}

\begin{proof}
As before we take the Rees algebra with respect to $V$. We get 
\[
\mathcal{R}_V(\mathbb{K}\otimes D)\simeq \mathbb{K}[\tilde{t_i},\theta]\langle\tilde{\partial_{x_i}},\tilde{\partial_{t_i}}\rangle
\]
and
 \[
A:=\textrm{gr}^{F}(\mathcal{R}_V(\mathbb{K}\otimes D))\simeq \mathbb{K}[\tilde{t_i},\theta,\tilde{\partial_{x_i}},\tilde{\partial_{t_i}}].
\]
The ring $A$ is bigraded as follows:
\begin{description}
\item
\quad $\textrm{deg}(\tilde{t_i})=(0,-1)$, \quad $\textrm{deg}(\theta)=(0,1)$,
\quad
$\textrm{deg}(\tilde{\partial_{x_i}})=(1,0)$,\quad
$\textrm{deg}(\tilde{\partial_{t_i}})=(1,1)$.
\end{description}
This is not a positive grading since $\mathbb{K}[(\tilde{t_i}\theta)]=A_{0,0}$ is infinite over $\mathbb{K}$. Let 
\[
\tilde{M}=\mathbb{K}\otimes \textrm{gr}^{F}(\mathcal{R}_V(M)).
\]
 A bifiltered free resolution of $M$ induces a bigraded free resolution of $\tilde{M}$, thus $K_{F,V}(M;T_1,T_2)=K(\tilde{M};T_1,T_2)$. 

Let us endow $M$ with another good bifiltration $(F_{d,k}^{'}(M))_{d,k}$. We denote by $M'$ the module $M$ endowed with this bifiltration. 
In view of Proposition \ref{prop2}, it is sufficient to prove
\begin{itemize}
\item $\textrm{rad}(\textrm{ann}\tilde{M})=\textrm{rad}(\textrm{ann}\tilde{M'})$
\item For any prime ideal $\mathfrak{p}$ of $A$,
 $\textrm{mult}_{\mathfrak{p}}\tilde{M}=\textrm{mult}_{\mathfrak{p}}\tilde{M'}$.
\end{itemize}

To prove these two assertions, we argue exactly in the same way as in the proof of Proposition 1.3.2 of \cite{laurent88}. For the convenience of the reader, we give here the details.

We shall also use the behaviour of dimensions and multiplicities in short exact sequences. 

\begin{lemme}[\cite{cimpa1}, Proposition 24]\label{lemme2}
Let $$0\to E\to F\to G\to 0$$
 be an exact sequence of finite type $A$-modules, and let $\mathfrak{p}$ be a prime ideal of $A$. Then
\begin{enumerate}
\item $\mathrm{dim} F_{\mathfrak{p}}=\mathrm{max}(\mathrm{dim} E_{\mathfrak{p}},\mathrm{dim} G_{\mathfrak{p}}).$
\item If $\mathrm{dim} E_{\mathfrak{p}}=\mathrm{dim} G_{\mathfrak{p}}$, then
$\mathrm{mult}_{\mathfrak{p}}F=\mathrm{mult}_{\mathfrak{p}}E+\mathrm{mult}_{\mathfrak{p}}G$.\\
If $\mathrm{dim} E_{\mathfrak{p}}<\mathrm{dim} G_{\mathfrak{p}}$, then
$\mathrm{mult}_{\mathfrak{p}}F=\mathrm{mult}_{\mathfrak{p}}G$.\\
If $\mathrm{dim} E_{\mathfrak{p}}>\mathrm{dim} G_{\mathfrak{p}}$, then
$\mathrm{mult}_{\mathfrak{p}}F=\mathrm{mult}_{\mathfrak{p}}E$.\\
\end{enumerate}
\end{lemme}

We will follow the proof of \cite{laurent88} and indicate at each step how to prove :
\begin{description}
\item[Claim 1.] $\textrm{rad}(\textrm{ann}\tilde{M})\subset\textrm{rad}(\textrm{ann}\tilde{M'})$,
\item[Claim 2.] 
$\textrm{mult}_{\mathfrak{p}}\tilde{M}\geq\textrm{mult}_{\mathfrak{p}}\tilde{M'}$
if $\textrm{dim} \tilde{M}_{\mathfrak{p}}=\textrm{dim} \tilde{M'}_{\mathfrak{p}}$. 
\end{description}

First, since $F_{d,k}(M)$ and $F'_{d,k}(M)$ are good bifiltrations, there exist $d_{0},k_{0}\in\mathbb{N}$ such that for any $d,k$,  $F_{d,k}(M)\subset F'_{d+d_{0},k+k_{0}}(M)$. Let us denote by $M''$ the module $M$ endowed with the bifiltration $(F'_{d+d_{0},k+k_{0}}(M))_{d,k}$. The algebraic cycle associated with $\tilde{M'}$ is equal to the algebraic cycle associated with $\tilde{M''}$. Thus we can suppose $F'_{d,k}(M)\subset F_{d,k}(M)$.

Let us introduce the Rees algebra $\mathcal{R}(D)$ with respect to the bifiltration $F,V$, i.e. 
\[
\mathcal{R}(D)=\bigoplus_{d,k}F_{d,k}(D)\tau^{d}\theta^{k}.
\]
This is isomorphic to the $\C$-algebra generated by $x_{i}$, $t_{i}\theta^{-1}$, $\partial_{x_{i}}\tau$, 
$\partial_{t_{i}}\tau\theta$, $\tau$ and $\theta$, subject to the relations 
$[\partial_{x_{i}}\tau,x_{i}]=\tau$ and $[\partial_{t_{i}}\tau\theta,t_{i}\theta^{-1}]=\tau$. This is a noetherian algebra.

We define also the Rees module $\mathcal{R}(M)=\bigoplus_{d,k}F_{d,k}(M)\tau^{d}\theta^{k}$.
We have 
\[
\textrm{gr}^{F}(\mathcal{R}_V(M))\simeq \frac{\mathcal{R}(M)}{\tau\mathcal{R}(M)}.
\]

Let us suppose moreover that there exists $r\geq 1$ such that for any $d,k$, 
$F'_{d,k}(M)\subset F_{d,k}(M)\subset F'_{d+r,k}(M)$. Let $F''_{d,k}(M)=F_{d,k}(M)\cap F'_{d+1,k}(M)$.
We have 
\[F'_{d,k}(M)\subset F''_{d,k}(M)\subset F'_{d+1,k}(M)\quad \textrm{and}\quad 
F_{d-r+1,k}(M)\subset F''_{d,k}(M)\subset F_{d,k}(M). 
\]
By induction on $r$ we can suppose $r=1$, i.e. 
$\tau\mathcal{R}(M)\subset \mathcal{R}(M')\subset \mathcal{R}(M)$. Then we have the following exact sequences
of $\textrm{gr}^{F}\mathcal{R}_{V}(D)$-modules of finite type:
\[
0\to \frac{\tau\mathcal{R}(M)}{\tau\mathcal{R}(M')}\to \frac{\mathcal{R}(M')}{\tau\mathcal{R}(M')}
\to\frac{\mathcal{R}(M')}{\tau\mathcal{R}(M)}\to 0
\]
\[
0\to \frac{\mathcal{R}(M')}{\tau\mathcal{R}(M)}\to \frac{\mathcal{R}(M)}{\tau\mathcal{R}(M)}
\to\frac{\mathcal{R}(M)}{\mathcal{R}(M')}\to 0.
\]
After tensorizing by $\mathbb{K}$, we deduce $\textrm{rad}(\textrm{ann}\tilde{M})=\textrm{rad}(\textrm{ann}\tilde{M'})$. 
Then using Lemma \ref{lemme2}, we get $\textrm{mult}_{\mathfrak{p}}\tilde{M}=\textrm{mult}_{\mathfrak{p}}\tilde{M'}$. 

Let $F''_{d,k}(M)=F_{d,k}(M)\cap (\cup_{i}F'_{i,k}(M))$. We have :
\[
\mathcal{R}(M'')=\mathcal{R}(M)\cap(\cup_{i\geq 0}\tau^{-i}\mathcal{R}(M')).
\]
Let $\mathcal{L}_{j}=\mathcal{R}(M)\cap (\cup_{0\leq i\leq j}\tau^{-i}\mathcal{R}(M'))$. This is an ascending chain of finite type sub-modules of $\mathcal{R}(M)$. Hence it is stationary and there exists an integer $r\geq 0$ such that
\[
\mathcal{R}(M'')=\mathcal{R}(M)\cap \tau^{-r}\mathcal{R}(M').
\]
In particular $\mathcal{R}(M'')$ is of finite type and $F''_{d,k}(M)$ is a good bifiltration.\\ 
We have $\tau^{r}\mathcal{R}(M'')\subset \mathcal{R}(M')\subset\mathcal{R}(M'')$, i.e.\
we are in the situation of the previous paragraph. This implies $\textrm{rad}(\textrm{ann}\tilde{M''})=\textrm{rad}(\textrm{ann}\tilde{M'})$ and $\textrm{mult}_{\mathfrak{p}}\tilde{M''}=\textrm{mult}_{\mathfrak{p}}\tilde{M'}$. 

On the other hand, we have a canonical injection
\[
\frac{\mathcal{R}(M'')}{\tau\mathcal{R}(M'')}\to \frac{\mathcal{R}(M)}{\tau\mathcal{R}(M)}.
\]
Then $\textrm{rad}(\textrm{ann}\tilde{M''})\subset\textrm{rad}(\textrm{ann}\tilde{M})$, and Claim 1 is proved. 
From this canonical injection, we deduce Claim 2 by using Lemma \ref{lemme2}. 
\end{proof}

\section{Nicely bifiltered $D$-modules}

In this section we consider a bifiltered $D$-module satisfying the following condition:

\begin{definition}
Let $M$ be a $D$-module endowed with a good bifiltration. We say that the bifiltration is \emph{nice} if for any $d,k$, 
\begin{equation}
\left(\bigcup_{d'}F_{d',k}(M)\right)\bigcap \left(\bigcup_{k'} F_{d,k'}(M)\right)=F_{d,k}(M).
\end{equation}
In such a case, we say that $M$ is nicely bifiltered.
\end{definition}

\begin{definition}
Let $N$ be a bigraded $\textrm{gr}^V(D^{(h)})$-module. $N$ is said to be \emph{$h$-saturated} if the map 
$N\to N$ sending $m$ to $hm$ is injective.\\
Let $N$ be a bigraded $\textrm{gr}^F(\mathcal{R}_V(D))$-module. $N$ is said to be \emph{$\theta$-saturated} if the map 
$N\to N$ sending $m$ to $\theta m$ is injective.
\end{definition}
\begin{lemme}\label{lemme8}
The following are equivalent :
\begin{enumerate}
\item $M$ is nicely bifiltered,
\item $\mathrm{gr}^{V}(\mathcal{R}_{F}(M))$ is $h$-saturated,
\item $\mathrm{gr}^{F}(\mathcal{R}_{V}(M))$ is $\theta$-saturated.
\end{enumerate}
\end{lemme}

\begin{proof}
By definition, 2) and 3) are equivalent to the following : $\forall d,k, F_{d+1,k}(M)\cap F_{d,k+1}(M) \subset F_{d,k}(M)$. By \cite{moi}, Lemma 1.1, this is equivalent to 1).
\end{proof}

\paragraph{$h$-saturatedness and Gr\"obner bases.}
Let us give a criterion for $h$-saturatedness using Gr\"obner bases. Using the preceding lemma, that leads to a criterion for the niceness of a bifiltration.
Let in this paragraph $D^{(h)}=\C[x_1,\dots,x_n]\langle\partial_1,\dots,\partial_n,h\rangle$. It is graded by setting for any $i$, $\deg x_i=0$, $\deg\partial_i=1$ and $\deg h=1$. 

Let $<''$ be a well-order on $\mathbb{N}^{2n}$, compatible with sums. 
Then we define a well-order $<'$ on $\mathbb{N}^{2n+1}$ by
\[
(\alpha,\beta,k)<'(\alpha',\beta',k') \quad\textrm{iff}\quad
\]
\[
\left\{
\begin{array}{l}
\vert\beta\vert+k <\vert\beta'\vert+k' \\
\textrm{or} \quad \vert\beta\vert+k =\vert\beta'\vert+k'
\quad\textrm{and}\quad
\vert\beta\vert<\vert\beta'\vert\\
\textrm{or} \quad \vert\beta\vert+k =\vert\beta'\vert+k',
\vert\beta\vert=\vert\beta'\vert\ \textrm{and}\ 
(\alpha,\beta)<''(\alpha',\beta').
\end{array}
\right.
\]
This is a well-order on the monomials of $D^{(h)}$ adapted to the $F$-filtration.
To deal with submodules of $(D^{(h)})^r$, we define a well-ordering $<$ on $\mathbb{N}^{2n+1}\times \{1,\dots,r\}$ by
 \[
(\alpha,\beta,k,i)<(\alpha',\beta',k',i') \quad\textrm{iff}\quad
\left\{
\begin{array}{l}
(\alpha,\beta,k)<'(\alpha',\beta',k') \\
\textrm{or} \quad (\alpha,\beta,k)=(\alpha',\beta',k')
\quad\textrm{and}\quad i<i'.
\end{array}
\right.
\]
Note that if $(\alpha,\beta,k,i)\geq (\alpha',\beta',k',i')$ and $\vert\beta\vert+k=\vert\beta'\vert+k'$, then $k\leq k'$.
If $P\in W^r$, we denote by $\textrm{in}(P)$ the leading monomial of $P$.

\begin{definition}
Let $P_1,\dots,P_s$ be a Gr\"obner base of a homogeneous submodule $N\subset (D^{(h)})^r$. Such a base is called \emph{minimal} if 
\[
\forall i,\ \mathrm{Exp}P_i\notin \bigcup_{j\neq i}\left(\mathrm{Exp}P_j+\mathbb{N}^{2n+1}\right).
\]
\end{definition}

\begin{prop}\label{prop11}
The following assertions are equivalent :
\begin{enumerate}
\item $(D^{(h)})^r/N$ is $h$-saturated.
\item For any minimal homogeneous Gr\"obner base $P_1,\dots,P_s$ of $N$, 
for any $i$, $h$ does not divide $\mathrm{in}\,P_i$.
\item There exists a minimal homogeneous Gr\"obner base $P_1,\dots,P_s$ of $N$, such that
for any $i$, $h$ does not divide $\mathrm{in}\,P_i$.
\end{enumerate}
\end{prop}

\begin{proof}
Let us prove $1)\Rightarrow 2)$
Let  $P_1,\dots,P_s$ be a minimal homogeneous Gr\"obner base of $N$. 
Suppose that there exists $i$ such that $h$ divides $\textrm{in}P_i$.
Then $h$ divides $P_i$ by the definition of $<$. 
By $h$-saturatedness, $P_i/h\in N$. Thus
\[
\textrm{Exp}\frac{P_i}{h}\in\bigcup_{j\neq i}\left(\textrm{Exp}P_j+\mathbb{N}^{2n+1}\right),
\]
then
\[
\textrm{Exp}P_i=h\textrm{Exp}\frac{P_i}{h}\in\bigcup_{j\neq i}\left(\textrm{Exp}P_j+\mathbb{N}^{2n+1}\right),
\]
which contradicts the minimality.

$2)\Rightarrow 3)$ is obvious. Let us show $3)\Rightarrow 1)$.
Let $P\in (D^{(h)})^r$ homogeneous such that $hP\in N$. We shall show that $P\in N$.
By division, $hP=\sum Q_iP_i$ with for any $i$, $Q_i\in D^{(h)}$ homogeneous, 
$\textrm{deg}(Q_iP_i)=\textrm{deg}(hP)$, and 
$\textrm{ord}^{F}(Q_iP_i)\leq\textrm{ord}^{F}(hP)$.

Let us suppose that there exists $i$ such that $h$ does not divide $Q_i$.
Then $\textrm{ord}^{F}Q_i=\textrm{deg}Q_i$. Since $h$ does not divide $P_i$,
we have $\textrm{ord}^{F}P_i=\textrm{deg}P_i$.
Then
\[
\textrm{ord}^{F}(Q_iP_i)=\textrm{ord}^{F}(Q_i)+\textrm{ord}^{F}(P_i)
=\textrm{deg}Q_i+\textrm{deg}P_i=\textrm{deg}(hP).
\]
But 
\[
\textrm{ord}^{F}(Q_iP_i)\leq\textrm{ord}^{F}(hP)<\textrm{deg}(hP),
\]
a contradiction. Thus for any $i$, $h$ divides $Q_i$ and $P=\sum (Q_i/h)P_i\in N$.
\end{proof}

We shall make a link between the $(F,V)$-multidegree and the theory of slopes of Y. Laurent, c.f. \cite{laurent04}.
We consider intermediate filtrations $L$ between $F$ and $V$, denoted by $pF+qV$ with $p>0$, $q>0$, defined by
\[
L_{r}(D)=\sum_{dp+kq\leq r}F_{d,k}(D).
\]
Similarly we endow $M$ with the $L$-filtration 
$
L_{r}(M)=\sum_{dp+kq\leq r}F_{d,k}(M),
$
which is a good filtration since taking a bifiltered free presentation 
\[
 D^{r_1}[\mathbf{n}^{(1)}][\mathbf{m}^{(1)}]\to 
D^{r_0}[\mathbf{n}^{(0)}][\mathbf{m}^{(0)}]\to M\to 0,
\]
we see that $\textrm{gr}^{L}(M)$ is isomorphic to a quotient of 
$\textrm{gr}^{L}(D^{r_0}[p\mathbf{n}^{(0)}+q\mathbf{m}^{(0)}])$. 

On the other hand, since $\textrm{gr}^{V}(M)$ is isomorphic to a quotient of 
$\textrm{gr}^{V}(D^{r_0}[\mathbf{m}^{(0)}])[\mathbf{n}^{(0)}]$, it is endowed with a natural $F$-filtration. Similarly, $\textrm{gr}^{F}(M)$ is isomorphic to a quotient of 
$\textrm{gr}^{F}(D^{r_0}[\mathbf{n}^{(0)}])[\mathbf{m}^{(0)}]$, and it is endowed with a natural $V$-filtration.

In \cite{moi}, we considered also the bigraded module
\[
\textrm{bigr}(M)=\bigoplus_{d,k}\frac{F_{d,k}(M)}{F_{d,k-1}(M)+F_{d-1,k}(M)}
\]
over the ring $\textrm{bigr}(D)\simeq\textrm{gr}^{V}(\textrm{gr}^{F}(D))\simeq \textrm{gr}^{F}(\textrm{gr}^{V}(D))$.

\begin{lemme}\label{lemme3}
If $M$ is nicely bifiltered, we have  
\[
\mathrm{bigr}(M)\simeq\mathrm{gr}^{V}(\mathrm{gr}^{F}(M))\simeq \mathrm{gr}^{F}(\mathrm{gr}^{V}(M)).
\]
\end{lemme}

\begin{proof}
For the sake of simplicity, we suppose that $\mathbf{n}^{(0)}=\mathbf{m}^{(0)}=\mathbf{0}$ and  consider $M=D^r/N$. We have 
\[
F_{d,k}(M)=\frac{F_{d,k}(D^r)+N}{N},
\quad  F_{d}(M)=\frac{F_{d}(D^r)+N}{N},
\quad V_{k}(M)=\frac{V_{k}(D^r)+N}{N}.
\]
The niceness assumption is equivalent to the following:
\begin{equation}\label{eq4}
\forall d,k,\ (F_{d}(D^r)+N)\cap (V_{k}(D^r)+N)\subset F_{d,k}(D^r)+N.
\end{equation}
We have $\textrm{gr}^V(M)=\textrm{gr}^V(D^r)/\textrm{gr}^V(N)$ with
\[
\textrm{gr}^V(N)=\bigoplus_k \frac{V_k(D^r)\cap N+V_{k-1}(D^r)}{V_{k-1}(D^r)}.
\]
We naturally define
\begin{eqnarray*}
F_d(\textrm{gr}^V(N)) & = & F_d(\textrm{gr}^V(D^r))\cap\textrm{gr}^V(N)\\
 & = & \bigoplus_k \frac{F_{d,k}(D^r)+V_{k-1}(D^r)}{V_{k-1}(D^r)}\bigcap 
 \frac{V_k(D^r)\cap N+V_{k-1}(D^r)}{V_{k-1}(D^r)}.
\end{eqnarray*}
Thus we have 
\[
\textrm{gr}^F\textrm{gr}^V (N)=\bigoplus_{d,k}
\frac{(F_{d,k}(D^r)+V_{k-1}(D^r))\cap (V_k(D^r)\cap N+V_{k-1}(D^r))}
{(F_{d-1,k}(D^r)+V_{k-1}(D^r))\cap (V_k(D^r)\cap N+V_{k-1}(D^r))}.
\]
This is included in 
\[
\textrm{gr}^F\textrm{gr}^V(D^r)=\bigoplus_{d,k}\frac{F_{d,k}(D^r)}{F_{d-1,k}(D^r)+F_{d,k-1}(D^r)}
\]
via the map
\[
(F_{d,k}(D^r)+V_{k-1}(D^r))\cap (V_k(D^r)\cap N+V_{k-1}(D^r))\to F_{d,k}(D^r)+V_{k-1}(D^r)
\to F_{d,k}(D^r).
\]
Hence 
\begin{eqnarray*}
\textrm{gr}^F\textrm{gr}^V (M) & = &
\textrm{gr}^F\textrm{gr}^V (D)/\textrm{gr}^F\textrm{gr}^V (N)\\
 & = & 
 \frac{F_{d,k}(D^r)}{F_{d,k}(D^r)\cap(V_k(N)+V_{k-1}(D^r))+F_{d-1,k}(D^r)+F_{d,k-1}(D^r)}.
\end{eqnarray*}
On the other hand,
\begin{eqnarray*}
\textrm{bigr}_{d,k}(M) & = & \frac{F_{d,k}(M)}{F_{d-1,k}(M)+F_{d,k-1}(M)}\\
 & = &  \frac{F_{d,k}(D^r)+N}{F_{d-1,k}(D^r)+F_{d,k-1}(D^r)+N}\\
  & = & \frac{F_{d,k}(D^r)}{F_{d,k}(D^r)\cap (F_{d-1,k}(D^r)+F_{d,k-1}(D^r)+N)}\\
  & = & \frac{F_{d,k}(D^r)}{F_{d-1,k}(D^r)+F_{d,k-1}(D^r)+N\cap F_{d,k}(D^r)}.
\end{eqnarray*}
We have to show
\begin{eqnarray}
F_{d-1,k}(D^r)+F_{d,k-1}(D^r)+N\cap F_{d,k}(D^r) & = & 
F_{d,k}(D^r)\cap(V_k(N)+V_{k-1}(D^r))\nonumber\\
 & + & F_{d-1,k}(D^r)+F_{d,k-1}(D^r).\label{eq5}
\end{eqnarray}
The inclusion $\subset$ is obvious. On the other hand, 
\begin{eqnarray*}
F_{d,k}(D^r)\cap(V_k(N)+V_{k-1}(D^r)) & = & F_{d,k}(D^r)\cap(N+V_{k-1}(D^r))\\
 & \subset & (F_{d,k-1}(D^r)+N)\cap F_{d,k}(D^r)\quad(\textrm{using}\,(\ref{eq4}))\\
  & \subset &  F_{d,k-1}(D^r)+N\cap F_{d,k}(D^r),
\end{eqnarray*}
which proves (\ref{eq5}).

We have showed that 
$\textrm{bigr}(M)\simeq\textrm{gr}^{F}(\textrm{gr}^{V}(M))$, and by exchanging the role of $F$ and $V$ we show that $\textrm{bigr}(M)\simeq\textrm{gr}^{V}(\textrm{gr}^{F}(M))$.

Note also that under the niceness assumption, the module $\textrm{bigr}(N)$ is identified with a submodule of $\textrm{bigr}(D^r)$ such that 
$\textrm{bigr}(M)\simeq\textrm{bigr}(D^r)/\textrm{bigr}(N)$.
\end{proof}

\begin{lemme}[\cite{SST}, Lemma 2.1.6]\label{lemme4}
For $\epsilon>0$ small enough, 
\[
\mathrm{gr}^{V}(\mathrm{gr}^{F}(M))\simeq \mathrm{gr}^{L}(M)\quad \textrm{with}\quad L=F+\epsilon V,
\]
 and 
\[
\mathrm{gr}^{F}(\mathrm{gr}^{V}(M))\simeq \mathrm{gr}^{L}(M)\quad \textrm{with} \quad L=V+\epsilon F.
\]
\end{lemme}

It is known that for any $L$, $\textrm{gr}^{L}(M)$ defines an algebraic cycle 
independent of the good filtration (the proof is almost the same as for the $F$-filtration). The variety defined by the annihilator of $\textrm{gr}^L(M)$ is denoted by $\textrm{char}_L(M)$. Remember that $\mathbb{K}$ denotes the fraction field of $\C[x]$. The module $\mathbb{K}\otimes\textrm{gr}^{L}(M)$ also defines an algebraic cycle independent of the good filtration.

\begin{prop}\label{prop14}
If $M$ is nicely bifiltered, we have 
\[
 K_{F,V}(M;T_{1},T_{2})=K(\mathrm{bigr}(M);T_{1},T_{2})=K(\mathrm{gr}^{L}M;T_{1},T_{2})
\]
with $L=V+\epsilon F$ or $L=F+\epsilon V$ with $\epsilon>0$ small enough. Here $\mathrm{gr}^{L}M$ is considered as a bigraded module.
\end{prop}

\begin{proof}
Under this assumption, any bifiltered free resolution of $M$ induces a bigraded free resolution of $\textrm{bigr}M$ (see \cite{moi}, Theorem 1.1, forgetting the minimality). Thus $K_{F,V}(M;T_{1},T_{2})=K(\textrm{bigr}M;T_{1},T_{2})$. But by Lemma \ref{lemme3} and Lemma \ref{lemme4}, $\textrm{bigr}M\simeq \textrm{gr}^{L}(M)$. 
\end{proof}

\paragraph{Remark 4.1.}
 The multidegree $\mathcal{C}_{F,V}(M;T_{1},T_{2})$ has total degree 
\[
d=\textrm{codim}\,\mathbb{K}\otimes\textrm{gr}^F(\mathcal{R}_V(M)),
\]
by definition. On the other hand, since the multigrading on $\mathbb{K}\otimes\textrm{bigr}\,D$ is positive, we know that the first non-zero terms in the expansion of $K_{F,V}(M;1-T_{1},1-T_{2})$ have total degree equal to 
\[
d'=\textrm{codim}\,(\mathbb{K}\otimes\textrm{bigr}\,M).
\]
Thus $d\leq d'$. If $d<d'$, then $\mathcal{C}_{F,V}(M;T_{1},T_{2})=0$.
We will see in the next section non trivial cases in which $d=d'$.

We then have, applying Proposition \ref{prop2} : 

\begin{theo}
The multidegree $\mathcal{C}_{F,V}(M;T_{1},T_{2})$ only depends on 
$\mathrm{codim}\,\mathbb{K}\otimes\mathrm{gr}^F(\mathcal{R}_V(M))$ and on the algebraic cycle defined by $\mathbb{K}\otimes\mathrm{gr}^{L}(M)$ with $L=V+\epsilon F$ or $L=F+\epsilon V$ with $\epsilon>0$ small enough.
\end{theo}

Let us recall some geometric meaning related to the $L$-filtration. 
Let $X=\C^{n+p}$, $Y=\{t=0\}\subset X$ and $\Lambda=T^*_Y X$ the conormal bundle. 
We have $\textrm{gr}^L(D)\simeq\mathcal{O}(T^*\Lambda)$, c.f. \cite{laurent04}. 
Let $\pi : T^*\Lambda\to Y$ be the canonical projection. 

By Proposition \ref{prop10}, 
$\mathcal{C}_{F,V}(\mathbb{K}\otimes\textrm{gr}^{L}(M);T_1,T_2)=\mathcal{C}_{F,V}(\textrm{gr}^{L}(M)^y;T_1,T_2)$ for $y\in Y$ generic. This depends only on the algebraic cycle on $\pi^{-1}(y)$ defined by $\textrm{gr}^{L}(M)^y$ for $y$ generic. $d'$ is equal to the codimension of $\textrm{char}_L(M)\cap \pi^{-1}(y)\subset \pi^{-1}(y)$, for $y$ generic.

For any $L$, we have $ \textrm{gr}^{L}(D)\simeq\textrm{gr}^{F}(\textrm{gr}^{V}(D))$ thus $ \textrm{gr}^{L}(D)$ is a bigraded ring. Following the theory of Y. Laurent, we say that $M$ has no slopes along $Y$ if for any $L$, the ideal $\textrm{rad}(\textrm{ann}\,\textrm{gr}^{L}(M))$ (defining $\textrm{char}_L (M)$) is bihomogeneous. The following means that niceness of the bifiltration is a strong regularity condition.

\begin{prop}
If $M$ is a nicely bifiltered holonomic $D$-module, then $M$ has no slopes along $Y$.
\end{prop}

\begin{proof}
As before, we identify $\mathcal{R}_V(D)$ with $D[\theta]$.
Let us take a bifiltered free presentation 

\begin{equation}\label{eq2}
D^{s}[\mathbf{n}][\mathbf{m}]\stackrel{\phi_{1}}{\to} D^{r}\stackrel{\phi_{0}}\to M\to 0,
\end{equation}
with $\phi_{1}(e_{i})=P^{(i)}=\sum_{j}P_{j}^{(i)}e_{j}$, and let $N=\textrm{Im}\,\phi_{1}$. For the sake of simplicity, we have assumed $\mathbf{n}^{(0)}=\mathbf{m}^{(0)}=\mathbf{0}$.
 This induces a bigraded free resolution
\[
\textrm{gr}^{F}(D[\theta])^{s}[\mathbf{n}][\mathbf{m}]\stackrel{\overline{\phi}_{1}}{\to}
 \textrm{gr}^{F}(D[\theta])^{r}\stackrel{\overline{\phi}_{0}}\to \textrm{gr}^{F}\mathcal{R}_{V}M\to 0.
\]
Using the lifting (\cite{granger04}, Proposition 2.7), we can suppose that the presentation (\ref{eq2}) is minimal, in the sense that the elements $\overline{\phi}_1(e_i)$ form a minimal set of generators of $\textrm{Ker}\,\overline{\phi}_{0}$.

Let us introduce some notations in order to determine $\overline{\phi}_1(e_i)$.\\
%
If $P=\sum a_{\nu,\mu}(x,\partial_x)t^{\nu}\partial_t^{\mu}\in V_k(D)$, we define
\[
H_k^V(P)=\sum a_{\nu,\mu}(x,\partial_x)t^{\nu}\partial_t^{\mu}\theta^{k-(|\mu|-|\nu|)}\in D[\theta],
\]
and $H^V(P)=H^V_{\textrm{ord}^V(P)}(P)$, the $V$-homogenization of $P$.
 Similarly if $P=\sum P_j e_j\in \oplus V_{m_j}(D)$, we define 
$H^V_{\mathbf{m}}(P)=\sum H^V_{m_j}(P_j)e_j\in (D[\theta])^r$.\\
Now if $P=\sum a_{\beta}(x,t,\partial_t,\theta)\partial_x^{\beta}\in F_d(D[\theta])$, we define
\[
\sigma^F_d(P)=\sum_{|\beta|=d} a_{\beta}(x,t,\partial_t,\theta)\partial_x^{\beta}\in\textrm{gr}^F_d(D[\theta]),
\]
and $\sigma^F(P)=\sigma^F_{\textrm{ord}^F(P)}P$. Similarly if $P=\sum P_j e_j\in \oplus F_{n_j}(D[\theta])$, we define 
$\sigma^F_{\mathbf{n}}(P)=\sum \sigma^F_{n_j}(P_j)e_j\in \textrm{gr}^F (D[\theta])^r$.

We have 
\[
\overline{\phi}_1(e_i)=\sigma^F_{\mathbf{n}}(H_{\mathbf{m}}(P)).
\]

For $P=\sum_{\nu,\beta,\mu}a_{\nu,\beta,\mu}(x)t^{\nu}\partial_{x}^{\beta}\partial_{t}^{\mu}$, let us define the Newton polygon by
\[
\mathcal{P}(P)=\bigcup_{(\nu,\beta,\mu)|a_{\nu,\beta,\mu}(x)\neq 0}(|\nu|-|\mu|,|\beta|+|\mu|)-\mathbb{N}^{2}\subset \mathbb{Z}^{2}.
\]
We say that $\mathcal{P}(P)$ is \emph{trivial} if it is equal to a translate of $(-\mathbb{N})\times(-\mathbb{N})$.

\newpage
For $1\leq i\leq s$, let $J(i)$ be the set of integers $1\leq j\leq r$ such that 
\begin{itemize}
\item $\textrm{ord}^{F}P^{(i)}_{j}=n_{i},$
\item $\textrm{ord}^{V}P^{(i)}_{j}=m_{i},$ 
\item $\mathcal{P}(P^{(i)}_{j})$ is trivial.
\end{itemize}
We claim that for any $i$, the set $J(i)$ is non-empty. 
Otherwise, $\theta$ would divide $\overline{\phi}_1(e_i)$. 
By $\theta$-saturatedness, $\overline{\phi}_1(e_i)/\theta$ would belong to $\textrm{gr}^{F}\mathcal{R}_{V}N$, thus the presentation (\ref{eq2}) would not be minimal.

Then $\textrm{bigr}N$ is generated by the elements 
\[
\sum_{j\in J(i)}\sigma^{F}\sigma^{V}(P^{(i)}_{j})e_{j}
\]
for $1\leq i\leq s$. Let $L$ be an intermediate filtration. We have 
\[
\sigma^{L}(P^{(i)})=\sum_{j\in J(i)}\sigma^{L}(P^{(i)}_{j})e_{j}=\sum_{j\in J(i)}\sigma^{F}\sigma^{V}(P^{(i)}_{j})e_{j}.
\]
Thus for any $L$,

\begin{equation}\label{eq3}
\textrm{bigr} N\subset \textrm{gr}^{L} N.
\end{equation}
If $\mathcal{M}$ is a $\textrm{gr}^L(D)$-module, we denote by $\textrm{supp} \mathcal{M}$ the zero-set of the annihilator of $\mathcal{M}$. 
By \cite{smith}, Theorem 1.1 and \cite{SST}, Theorem 2.2.1 (valid for any $L$), $\textrm{char}_{L}(M)=\textrm{supp} (\textrm{gr}^{L}(M))$ is pure of dimension $n+p$ for any $L$. 
Since $\textrm{bigr} N=\textrm{gr}^{F}\textrm{gr}^{V}(N)=\textrm{gr}^{L}(N)$ for $L$ close to $V$, 
then $\textrm{supp} (\textrm{bigr} M)$ is pure of dimension $n+p$. 

By (\ref{eq3}), we have for any $L$, 
$\textrm{char}_{L}M\subset \textrm{supp} (\textrm{bigr} M)$, 
thus $\textrm{char}_{L}M$ is the union of some irreducible components of $\textrm{supp} (\textrm{bigr} M)$. 
The irreducible components are bihomogeneous (a bihomogeneous module admits a bihomogeneous primary decomposition), so $\textrm{char}_{L}M$ is bihomogeneous. 
\end{proof}

\section{Examples from the theory of hypergeometric systems }

Let $D=\C[x_1,\dots,x_n]\langle \partial_1,\dots,\partial_n\rangle$.
We consider the $A$-hypergeometric $D$-module $M_A(\beta)=D/H_A(\beta)$. This is a holonomic system associated with a $d\times n$ integer matrix $A$ and $\beta_1,\dots,\beta_d\in\C$ as follows. We suppose that the abelian group generated by the columns $a_1,\dots,a_n$ of $A$ is equal to $\mathbb{Z}^d$. Let $I_A$ be the ideal of $\C[\partial_1,\dots,\partial_n]$ generated by the elements $\partial^u-\partial^v$ with $u,v\in\mathbb{N}^n$ such that $A.u=A.v$. 
The hypergeometric ideal $H_A(\beta)$ is the ideal of $D$ generated by $I_A$ and the elements $\sum_j a_{i,j}x_j\partial_j-\beta_i$ for $i=1,\dots,d$. The hypergeometric modules were introduced by I. M. Gelfand, M. M. Kapranov and A. V. Zelevinsky in \cite{GKZ}; their holonomicity (in the general case) was proved by A. Adolphson in \cite{ado}.

We endow $M$ with the quotient $F$-filtration and the quotient $V$-filtration with respect to  $x_1=\dots=x_n=0$. 

Let us assume that the abelian group generated by the rows of $A$ contains a vector $\mathbf{w}=(w_1,\dots,w_n)\in \mathbb{Z}^{n}_{>0}$. That is equivalent to the fact that the semigroup generated by the columns of $A$ is pointed. 
By applying the weight vector $W=(-\mathbf{w},\mathbf{w})$ to 
$(x,\partial)$, we get a grading on $D$. The hypergeometric module $M_A(\beta)$ is homogeneous w.r.t. to $W$.

Our first topic is to strenghten the correspondence between $\mathcal{C}_{F,V}(M_A(\beta);T_1,T_2)$ and $\mathcal{C}(\textrm{bigr}M_A(\beta);T_1,T_2)$, i.e. to prove that the modules $\textrm{bigr}M_A(\beta)$ and \\
 $\textrm{gr}^F(\mathbf{R}_V(M_A(\beta)))$ have the same codimension if $M_A(\beta)$ is nicely bifiltered.

The codimension of a finite type $D$-module $M$ is by definition the codimension of $gr^F(M)$, that does not depend on the good $F$-filtration. In fact we can make the weight vector vary as well.

\begin{prop}[\cite{SST}, pp. 65-66]\label{prop13}
Let $(\mathbf{u},\mathbf{v})\in \mathbb{N}^{2n}$ be a weight vector such that for all $i$, $u_i+v_i>0$. Endow $M$ with a good $(\mathbf{u},\mathbf{v})$-filtration. 
Then $\mathrm{codim}(\mathrm{gr}^{(\mathbf{u},\mathbf{v})}(M))=\mathrm{codim}M$.
\end{prop}

We have an analogous statement for $D^{(h)}$-modules, proved in the same way.
Let $(\mathbf{u},\mathbf{v},t)\in\mathbb{N}^{2n+1}$ such that for all $i$, $u_i+v_i>t$. Then $\textrm{gr}^{(\mathbf{u},\mathbf{v},t)}(D^{(h)})$ is commutative.

\begin{definition}
Let $M$ be a graded $D^{(h)}$-module of finite type. Endow $M$ with a good $(\mathbf{u},\mathbf{v},t)$-filtration. 
We define $\mathrm{codim} M=\mathrm{codim}(\mathrm{gr}^{(\mathbf{u},\mathbf{v},t)}M)$. This depends neither on the good filtration nor on the weight vector $(\mathbf{u},\mathbf{v},t)$.
\end{definition}

Finally, since $\textrm{gr}^V(D^{(h)})\simeq D^{(h)}$, we define in the same way the codimension of a $\textrm{gr}^V(D^{(h)})$-module of finite type.

We adopt the following notation. If $P=\sum a_{\beta}(x)\partial_x^{\beta}\in F_d(D)$, we define
$H_d(P)=\sum a_{\beta}(x)\partial_x^{\beta}h^{d-|\beta|}\in D^{(h)}$,
and the $F$-homogenization $H(P)=H_{\textrm{ord}^F(P)}(P)$. If $I$ is an ideal of $D$, let $H(I)$ be the ideal of $D^{(h)}$ generated by the elements $H(P)$ such that $P\in I$.
We have $\mathcal{R}_F(M)=D^{(h)}/H(I)$. 
Similarly we define the $V$-homogenization, denoted by $H^V(P)\in D[\theta]$ and $H^V(I)\subset D[\theta]$.

\begin{prop}\label{prop15}
Let $M=D/I$ be a $W$-homogeneous nicely bifiltered $D$-module.
Then the modules $M$, $\mathrm{gr}^F(\mathbf{R}_V(M))$, $\mathrm{gr}^V(\mathbf{R}_F(M))$ and $\mathrm{bigr} M$ all have the same codimension.
\end{prop}

\begin{proof}
First, we prove that 
\[
\textrm{codim} \mathcal{R}_F(M)=\textrm{codim} M.
\]

Let $<$ be a well-order on $\mathbb{N}^{2n}$ (the monomials of $D$) adapted to $F$, i.e. for any $\alpha,\alpha',\beta,\beta', |\beta|<|\beta'| \Rightarrow (\alpha,\beta)<(\alpha',\beta')$. We derive from it a well-order $<'$ on $\mathbb{N}^{2n+1}$ (the monomials of $(D^{(h)})$) in the following way:
\[
(\alpha,\beta,k)<'(\alpha',\beta',k') \quad\textrm{iff}\quad
\left\{
\begin{array}{l}
|\beta| +k<|\beta'|+k'\\
\textrm{or} \quad 
\left\{
\begin{array}{l}
|\beta| +k=|\beta'|+k'\\
\textrm{and}\quad (\alpha,\beta)<(\alpha',\beta'),
\end{array}
\right.
\end{array}
\right.
\]
which is adapted to the $F$-filtration.
Let $P_1,\dots,P_s$ be a Gr\"obner base of $I$ with respect to $<$. Then $H(P_1),\dots,H(P_s)$ is a Gr\"obner base of $H(I)$ with respect to $<'$ (use the Buchberger criterion).
We have $\sigma^F(H(P_i))=\sigma^F(P_i)\in \C[x,\xi]$, thus 
$\textrm{codim}(\textrm{gr}^F(\mathbf{R}_F(M)))=\textrm{codim}(\textrm{gr}^F(M))$.

Now, we prove that  
\[
\textrm{codim}(\textrm{gr}^V(\mathcal{R}_F(M)))=\textrm{codim} M.
\]
The module $\mathcal{R}_F(M)$ is bihomogeneous with respect to the weight vectors 
$(-\mathbf{w},\mathbf{w},0)$ and $(\mathbf{0},\mathbf{1},1)$.
Let $\mu=\textrm{max}(w_i-1)\in\mathbb{N}$ and
\[
\Lambda=(-\mathbf{1},\mathbf{1},0)-(-\mathbf{w},\mathbf{w},0)+\mu.(\mathbf{0},\mathbf{1},1)=
(\mathbf{w}-\mathbf{1},(1+\mu)\mathbf{1}-\mathbf{w},\mu.\mathbf{1})\in\mathbb{N}^{2n+1}.
\] 
Using the bihomogeneity, a $V$-adapted base of $H(N)$ is also adapted to $\Lambda$, so
$\textrm{gr}^{\Lambda}(\mathcal{R}_F(M))=\textrm{gr}^{V}(\mathcal{R}_F(M))$.
Then
\begin{eqnarray*}
\textrm{codim}\, \textrm{gr}^V(\mathcal{R}_F(M)) & = &
\textrm{codim}\,\textrm{gr}^{(\mathbf{0},\mathbf{1},0)}\textrm{gr}^V(\mathcal{R}_F(M))
\quad \textrm{(by definition)}\\
 & = & \textrm{codim}\,\textrm{gr}^{(\mathbf{0},\mathbf{1},0)}\textrm{gr}^{\Lambda}(\mathcal{R}_F(M))\\
 & = & \textrm{codim}\,\textrm{gr}^{\Lambda+\epsilon.(\mathbf{0},\mathbf{1},0)}(\mathcal{R}_F(M)) \quad\textrm{with}\ \epsilon>0, 
\end{eqnarray*}
by \cite{SST}, Lemma 2.1.6, which proves our assertion since 
$\Lambda+\epsilon.(\mathbf{0},\mathbf{1},0)\in\mathbb{N}^{2n+1}$.

Next, let us see that  
\[
\textrm{codim}(\textrm{gr}^F(\mathcal{R}_V(M)))=\textrm{codim} M.
\]
We will slightly modify the problem using the niceness assumptiom. We can endow $\textrm{gr}^F(D)\simeq\C[x,\xi]$ with a filtration with respect to the weight vector $(-\mathbf{1},\mathbf{1})$, which we still call the $V$-filtration. The module $\textrm{gr}^F(M)\simeq\textrm{gr}^F(D)/\textrm{gr}^F(I)$ is naturally endowed with the quotient $V$-filtration. In the same way as in the proof of Lemma \ref{lemme3}, we have 
\[
\textrm{gr}^F(\mathcal{R}_V(M))=\mathcal{R}_V(\textrm{gr}^F(M)).
\]
Thus we are reduced to show 
$\textrm{codim}(\mathcal{R}_V(\textrm{gr}^F(M))=\textrm{codim} M$.
As before, let $\mu=\textrm{max}(w_i-1)$ and define
$\Lambda=V-(-\mathbf{w},\mathbf{w})+\mu.(\mathbf{0},\mathbf{1})\in\mathbb{N}^{2n}$. 
We have a ring isomorphism 
\[
\mathcal{R}_V(\textrm{gr}^F(D))\simeq \textrm{gr}^F(D)[\theta]
\simeq \mathcal{R}_{\Lambda}(\textrm{gr}^F(D)),
\]
and 
$\mathcal{R}_V(\textrm{gr}^F(M))\simeq\mathcal{R}_{\Lambda}(\textrm{gr}^F(M))$ above this ring isomorphism.  Next, 
\begin{eqnarray*}
\textrm{codim}\,\mathcal{R}_{\Lambda}(\textrm{gr}^F(M)) & = & 
\textrm{codim}\,\textrm{gr}^{\Lambda}\textrm{gr}^F(M)\\
& = & 
\textrm{codim}\,\textrm{gr}^{F+\epsilon\Lambda}(M)\\
& = & \textrm{codim}\,M.
\end{eqnarray*}

Finally, we show that
\[
\textrm{codim}(\textrm{bigr}M)=\textrm{codim}(M).
\]
We have $\textrm{bigr}M\simeq\textrm{gr}^V\textrm{gr}^F(M)$, by Lemma \ref{lemme3}. Taking again
$\Lambda=V-(-\mathbf{w},\mathbf{w})+\mu.(\mathbf{0},\mathbf{1})$, the assertion follows from $\textrm{gr}^V\textrm{gr}^F(M)=\textrm{gr}^{\Lambda}\textrm{gr}^F(M)=
\textrm{gr}^{F+\epsilon\Lambda}(M)$.
\end{proof}

\begin{remark} If $M_A(\beta)$ is nicely bifiltered, then we have 
\[
\mathcal{C}_{F,V}(M_A(\beta);T_1,T_2)_{|T_2=0}=(\mathrm{rank}\, M_A(\beta)).T_1^n.
\]
\end{remark}
Indeed, a bifiltered free resolution induces a $F$-filtered free resolution, 
thus $K_{F}(M;T_1)=K_{F,V}(M;T_1,T_2)_{|T_2=1}$, so 
$K_{F}(M;1-T_1)=K_{F,V}(M;1-T_1,1-T_2)_{|T_2=0}$, and by the Proposition above, we have $\textrm{codim}\,\textrm{gr}^F(\mathcal{R}_V(M_A(\beta)))=\textrm{codim}\,M=\textrm{codim}\,\textrm{gr}^F(M_A(\beta))=n$. We conclude by using Proposition \ref{prop6}.

Let us note for $1\leq i\leq d$, $(Ax\xi)_i=\sum_j a_{i,j}x_j\xi_j\in\textrm{gr}^F(D)$.

\begin{lemme}\label{lemme7}
If $\mathrm{gr}^F(\C[\partial]/I_A)$ is Cohen-Macaulay, then $(Ax\xi)_1,\dots,(Ax\partial)_d$ is a regular sequence in $\mathrm{gr}^F(D/D I_A)$.
\end{lemme}

\begin{proof}
By \cite{stu05}, Proposition 7.5, $\textrm{dim} (C[\partial]/I_A)=d$. Using Proposition \ref{prop13}, we get 
$\textrm{dim}(\mathrm{gr}^F(D/D I_A))=n+d$. But $\textrm{dim}(\C[x,\xi]/(Ax\xi+\textrm{gr}^F(I_A))=n$ by \cite{SW}, proof of Proposition 3.8. The results follows from the Cohen-Macaulay assumption.
\end{proof}

\subsection{The homogeneous case}

We suppose moreover that the columns of $A$ lie in a common hyperplane, i.e. $(1,\dots,1)$ belongs to the $\mathbb{Q}$-row span of $A$. Then $I_A$ is homogeneous for the weight vector $(1,\dots,1)$ and $M_A(\beta)$ is $V$-homogeneous.

\begin{lemme}\label{lemme9}
$M_A(\beta)$ is nicely bifiltered.
\end{lemme}

Indeed, $M_A(\beta)$ is $V$-homogeneous, thus $\mathbf{R}_F(M_A(\beta))$ is also $V$-homogeneous, thus $\textrm{gr}^V\mathbf{R}_F(M_A(\beta))\simeq \mathbf{R}_F(M_A(\beta))$ is $h$-saturated. Then apply Lemma \ref{lemme8}.

\begin{lemme}\label{lemme5}
Let $R=\mathrm{bigr} D$ and $M$ be a finite type bigraded $R$-module. Let $P\in R$ be bihomogeneous of degree $(d,k)$. If $P$ is a non zero-divisor on $M$ then
\begin{enumerate}
\item $K_{F,V}(M/PM;T_1,T_2)=(1-T_1^dT_2^k)K_{F,V}(M;T_1,T_2)$ and
\item $\mathcal{C}_{F,V}(M/PM;T_1,T_2)=(dT_1+kT_2)\mathcal{C}_{F,V}(M;T_1,T_2)$.
\end{enumerate}
\end{lemme}

\begin{proof}
Let us prove 1). 
If $N$ is a bigraded $R$-module, let $S_{d,k}(N)$ be the bigraded module defined by 
$(S_{d,k}(N))_{d',k'}=N_{d'-d,k'-k}$. In particular, 
$S_{d,k}(D^r[\mathbf{n}][\mathbf{m}])=D^r[\mathbf{n}+d.\mathbf{1}][\mathbf{m}+k.\mathbf{1}]$. 
A bigraded free resolution
\[
\cdots\to\mathcal{L}_1\to\mathcal{L}_0\to M\to 0
\]
 of $M$ induces a bigraded free resolution
\[
\cdots\to S_{d,k}(\mathcal{L}_1)\to S_{d,k}(\mathcal{L}_0)\to S_{d,k}(M)\to 0
\]
 of $S_{d,k}(M)$.
We have a bigraded exact sequence 
\[
0\to S_{d,k}(M)\stackrel{P.}{\to}M\to\frac{M}{PM}\to 0.
\]
Then taking the cone of the morphism of resolutions 
$S_{d,k}(\mathcal{L}_{\bullet})\stackrel{P.}{\to} \mathcal{L}_{\bullet}$ gives a resolution 
\[
\cdots\to S_{d,k}(\mathcal{L}_1)\oplus\mathcal{L}_2\to 
S_{d,k}(\mathcal{L}_0)\oplus\mathcal{L}_1\to
\mathcal{L}_0\to \frac{M}{PM}\to 0
\]
of $M/PM$.
Then 1) follows, and 2) follows from 1).
\end{proof}

Let us denote by $\mathrm{vol}(A)$ the normalized volume of the convex hull in $\mathbb{R}^d$ of the set $\{0,a_1,\dots,a_n\}$. The normalization means that the set $[0,1]\times\dots\times[0,1]\subset\mathbb{R}^d$ has volume 
$d!$.

\begin{theo}\label{theo2}
If $\C[\partial]/I_A$ is homogeneous and Cohen-Macaulay, then for any $\beta\in\C^d$ we have
\[
\mathcal{C}_{F,V}(M_A(\beta);T_1,T_2)=
\mathrm{vol}(A).\sum_{j=d}^{n}\binom{n-d}{j-d}T_1^jT_2^{n-j},
\]
\end{theo}

\begin{proof}
By Proposition \ref{prop14}, Proposition \ref{prop15} and Lemma \ref{lemme9},
$\mathcal{C}_{F,V}(M_A(\beta);T_1,T_2)$ is equal to the sum of the terms of least degree in $K_{F,V}(\textrm{bigr}M_A(\beta);1-T_1,1-T_2)$, and by Lemma \ref{lemme3} we have 
\[
\textrm{bigr}M_A(\beta)\simeq\textrm{gr}^F\textrm{gr}^V(M_A(\beta))=\textrm{gr}^F(M_A(\beta)).
\]
When $\C[\xi]/I_A$ is Cohen-Macaulay, $(Ax\xi)_1,\dots,(Ax\xi)_d$ form a regular sequence in 
$\C[x,\xi]/I_A$, and $\textrm{gr}^F(H_A(\beta))$ is generated by $I_A$ and $(Ax\xi)_1,\dots,(Ax\xi)_d$, by Lemma \ref{lemme7} and \cite{SST}, Theorem 4.3.8.
Using Lemma \ref{lemme5} repeatedly, we get 
\[
\mathcal{C}_{F,V}(\textrm{gr}^F(M_A(\beta));T_1,T_2)=T_1^d.\mathcal{C}_{F,V}(\C[x,\xi]/I_A;T_1,T_2).
\] 
But $\mathcal{C}_{F,V}(\C[x,\xi]/I_A;T_1,T_2)=\mathcal{C}_{F,F}(\C[\xi]/I_A;T_1,T_2)$ since $I_A\subset \C[\xi]$. 
Let $R=\C[\xi]$, $P(T_1,T_2)=K_{F,F}(R/I_A;T_1,T_2)$ and 
 $Q(T)=K_{F}(R/I_A;T)$. Consider a graded free resolution 
\[
0\to R^{r_{\delta}}[\mathbf{n}^{(\delta)}]\to\cdots R^{r_{0}}[\mathbf{n}^{(0)}]\to R/I_A\to 0
\]
of $R/I_A$.
Then we have a bigraded free resolution 
\[
0\to R^{r_{\delta}}[\mathbf{n}^{(\delta)}][\mathbf{n}^{(\delta)}]\to
\cdots R^{r_{0}}[\mathbf{n}^{(0)}][\mathbf{n}^{(0)}]\to R/I_A\to 0
\]
of $R/I_A$. 
We deduce that $P(T_1,T_2)=Q(T_1T_2)$. We have $Q(1-T)=b_{n-d}T^{n-d}+O(n-d+1)$, with 
$b_{n-d}=\textrm{deg}(R/I_A)\neq 0$, and $O(n-d+1)$ denotes a polynomial of valuation greater than $n-d$. By \cite{GKZ94}, Chapter 6, Theorem 2.3, $\textrm{deg}(R/I_A)=\textrm{vol}(A)$.
We have
\begin{eqnarray*}
P(1-T_1,1-T_2) & = & Q((1-T_1)(1-T_2))\\
 & = & Q(1-(T_1+T_2-T_1T_2))\\
 & = & b_{n-d} (T_1+T_2)^{n-d}+O(n-d+1)\\
 & = &  b_{n-d}\left(\sum_{i=0}^{n-d}\binom{n-d}{i}T_1^iT_2^{n-d-i}\right)+O(n-d+1),
\end{eqnarray*}
from which the statement follows.
\end{proof}

To compute the multidegree in the following examples, we used the computer algebra systems Singular \cite{singular} and Macaulay2 \cite{M2}.

\paragraph{Example 1.} Let 
$A=\left(
\begin{array}{ccc}
1 & 1& 1\\
0 & 1 & 2
\end{array}
\right).
$
Then $I_A$ is generated by $\partial_1\partial_3-\partial_2^2$. For all $\beta$,
$\mathcal{C}_{F,V}(M_A(\beta);T_1,T_2)=2T_1^3+2T_1^2T_2$.

\paragraph{Example 2.} Let 
$
A=\left(
\begin{array}{cccc}
1 & 1& 1 & 1\\
0 & 1 & 2 & 3
\end{array}
\right).
$
Then $I_A$ is generated by 
$\partial_2\partial_4-\partial_3^2,
\partial_1\partial_4-\partial_2\partial_3,
\partial_1\partial_3-\partial_2^2.$
For all $\beta$,
$\mathcal{C}_{F,V}(M_A(\beta);T_1,T_2)=3T_1^4+6T_1^3T_2+3T_1^2T_2^2.$
\paragraph{} Let us give homogeneous non-Cohen-Macaulay examples from the book \cite{SST}.
Using Proposition \ref{prop10} repeatedly, we can establish the existence of a stratification of the space of the parameters $\beta_1,\beta_2$ by the multidegree. In the following two examples, this stratification equals the stratification by the holonomic rank.  
\paragraph{Example 3.} Let 
$
A=\left(
\begin{array}{cccc}
1 & 1& 1 & 1\\
0 & 1 & 3 & 4
\end{array}
\right).
$
Then $I_A$ is generated by 
$\partial_2\partial_4^2-\partial_3,\partial_1\partial_4-\partial_2\partial_3,\partial_1\partial_3^2-\partial_2^2\partial_4,\partial_1^2\partial_3-\partial_2^3$.
For $(\beta_1,\beta_2)\neq (1,2)$, we have 
\[
\mathcal{C}_{F,V}(M_A(\beta);T_1,T_2)=4T_1^4+8T_1^3T_2+4T_1^2T_2^2.
\]
For $(\beta_1,\beta_2)= (1,2)$, we have 
\[
\mathcal{C}_{F,V}(M_A(\beta);T_1,T_2)=5T_1^4+12T_1^3T_2+10T_1^2T_2^2+4T_1T_2^3+T_2^4.
\]
\paragraph{Example 4.} Let 
$
A=\left(
\begin{array}{ccccc}
1 & 1& 1 & 1 & 1\\
0 & 2 & 4 & 7 & 9
\end{array}
\right).
$
Then $I_A$ is generated by 
$\partial_2\partial_4-\partial_3^2,\partial_1^2-\partial_2\partial_3$.
Let $E=\{(2,10),(2,12),(3,19)\}.$
For $(\beta_1,\beta_2)\notin E$, we have 
\[
\mathcal{C}_{F,V}(M_A(\beta);T_1,T_2)=9T_1^5+27T_1^4T_2+27T_1^3T_2^2+9T_1^2T_2^3.
\]
For $(\beta_1,\beta_2)\in E$, we have 
\[
\mathcal{C}_{F,V}(M_A(\beta);T_1,T_2)=10T_1^5+32T_1^4T_2+37T_1^3T_2^2+19T_1^2T_2^3+5T_1T_2^4+T_2^5.
\]

\subsection{The inhomogeneous case}

%
%

Following arguments in the book \cite{SST}, we extend Theorem \ref{theo2} in the inhomogeneous case, for generic parameters $\beta$.

\begin{theo}\label{theo3}
Assume that $\C[\partial,h]/H(I_A)$ is Cohen-Macaulay. Then for generic $\beta$, the module $M_A(\beta)$ is nicely bifiltered and 
\[
\mathcal{C}_{F,V}(M_A(\beta);T_1,T_2)=
\mathrm{vol}(A).\sum_{j=d}^{n}\binom{n-d}{j-d}T_1^jT_2^{n-j}.
\]
\end{theo}

Here, the assumption is that the closure of the variety defined by $I_A$ in the projective space $\mathbb{P}^n$ is Cohen-Macaulay.

\begin{proof}
First, note that the $\C[\partial,h]$-module $\C[\partial,h]/H(I_A)$ and the $\textrm{gr}^F(\C[\partial])$-module
$\textrm{gr}^F(\C[\partial])/\textrm{gr}^F(I_A)$ have same codimension and same projective dimension. Thus by \cite{eisenbud95}, Corollary 19.15, the Cohen-Macaulayness of the former is equivalent to that of the latter.

Also,
\[
\mathcal{C}(\textrm{gr}^F(\C[\partial])/\textrm{gr}^F(I_A);T)=
\mathcal{C}(\C[\partial,h]/H(I_A);T)
=\textrm{deg}(\C[\partial,h]/H(I_A))T^{n-d}
\]
and again by \cite{GKZ94}, Chapter 6, Theorem 2.3, 
$\textrm{deg}(\C[\partial,h]/H(I_A))=\mathrm{vol}(A)$.

For generic $\beta$, by \cite{SST}, Theorem 3.1.3 (with $w=(1,\dots,1)$), and \cite{oaku01b}, Theorem 2.5,
\[
H^V(H_A(\beta))=D[\theta]H^V(I_A)+\sum_i D[\theta]((Ax\partial)_i-\beta_i).
\]
By Lemma \ref{lemme7}, because of the Cohen-Macaulay assumption, $(Ax\xi)_1,\dots,(Ax\xi)_d$ is a regular sequence in $\textrm{gr}^F(D[\theta])/\textrm{gr}^F(I_A)=
\textrm{gr}^F(D[\theta])/\textrm{gr}^F(H^V(I_A))$. That implies that $H^V(I_A)$ and $((Ax\partial)_i-\beta_i)_i$ form an $F$-involutive base of $H^V(H_A(\beta))$ (see \cite{SST}, Proposition 4.3.2). Then 
\begin{eqnarray*}
\textrm{gr}^F(H^V(H_A(\beta))) & = &
\textrm{gr}^F(D[\theta])\textrm{gr}^F(H^V(I_A))+
\sum_i \textrm{gr}^F(D[\theta])(Ax\partial)_i.\\
& = & \textrm{gr}^F(D[\theta])\textrm{gr}^F(I_A)+
\sum_i \textrm{gr}^F(D[\theta])(Ax\partial)_i.
\end{eqnarray*}
Thus $\textrm{gr}^F(H^V(H_A(\beta)))$ is generated by elements independent of $\theta$; this implies that $\textrm{gr}^F(\mathcal{R}_V(M_A(\beta)))$ is $\theta$-saturated (consider the graduation given by the degree in $\theta$), which is equivalent to niceness by Lemma \ref{lemme8}.

We have again $\textrm{bigr}M_A(\beta)\simeq\textrm{gr}^F\textrm{gr}^V(M_A(\beta))$. With same arguments as above, we show that $\textrm{gr}^F\textrm{gr}^V(H_A(\beta))$ is generated by $\textrm{gr}^F(I_A)$ and $(Ax\xi)_i$ for generic $\beta$. We conclude the computation of the multidegree as in the proof of Theorem \ref{theo2}.
\end{proof}
To finish, let us give examples in the inhomogeneous case.


\paragraph{Example 5.} 
Let
$
A=\left(
\begin{array}{ccc}
0 & 1 & 3  \\
4 & 3 & 2
\end{array}
\right).
$
Then $I_A$ is generated by 
$\partial_1^7\partial_3^4-\partial_2^{12}$. The ring $\C[\partial,h]/H(I_A)$ is Cohen-Macaulay.
For any $\beta$, $M_A(\beta)$ is nicely bifiltered and
$\mathcal{C}_{F,V}(M_A(\beta);T_1,T_2)=12T_1^3+12T_1^2T_2.$

\paragraph{Example 6.} 
Let
$
A=\left(
\begin{array}{cccc}
-2&-1 & 0 & 1  \\
1 & 1 & 2 & 2
\end{array}
\right).
$
Then $I_A$ is generated by 
$\partial_2^2\partial_4^2-\partial_3^3, 
\partial_1\partial_4-\partial_2\partial_3,
\partial_1\partial_3^2-\partial_2^3\partial_4,
\partial_1^2\partial_3-\partial_2^4$.
The ring $\C[\partial,h]/H(I_A)$ is not Cohen-Macaulay.
For $\beta$ generic, $M_A(\beta)$ is nicely bifiltered and
\[
\mathcal{C}_{F,V}(M_A(\beta);T_1,T_2)=6T_1^4+12T_1^3T_2+6T_1^2T_2^2.
\]
We could check that the couple $\beta=(-1,2)$ is exceptional. In that case $M_A(\beta)$ is also nicely bifiltered and we have
\[
\mathcal{C}_{F,V}(M_A(\beta);T_1,T_2)=7T_1^4+16T_1^3T_2+12T_1^2T_2^2+4T_1T_2^3+T_2^4.
\]

Let us remark that in Examples 1--6, the formula of Theorems \ref{theo2} and \ref{theo3} holds for generic $\beta$, sometimes without the Cohen-Macaulay assumption.

\end{document}